\documentclass[11pt, reqno]{amsart}

\usepackage{amssymb,latexsym,amsmath,amsfonts}
\usepackage{mathrsfs}
\usepackage{graphicx}
\usepackage[usenames]{color}
\usepackage{enumitem}
\allowdisplaybreaks

\numberwithin{equation}{section}

\theoremstyle{definition}
\newtheorem{definition}{Definition}[section]

\theoremstyle{remark}
\newtheorem{remark}[definition]{Remark}

\theoremstyle{plain}
\newtheorem{theorem}[definition]{Theorem}
\newtheorem{result}[definition]{Result}
\newtheorem{lemma}[definition]{Lemma}
\newtheorem{proposition}[definition]{Proposition}

\newtheorem{fact}[definition]{Fact}

\newcommand{\eps}{\varepsilon}

\newcommand{\zt}{\zeta}

\newcommand{\zahl}{\mathbb{Z}}  
\newcommand{\nat}{\mathbb{N}}
\newcommand{\excep}{\mathcal{E}}

\newcommand{\fami}{\mathscr{F}}
\newcommand{\fat}{\boldsymbol{{\sf F}}}
\newcommand{\jul}{\boldsymbol{{\sf J}}}
\newcommand{\Lb}{\boldsymbol{{\sf m}}}

\newcommand{\bdy}{\partial}

\newcommand{\OM}{\Omega}

\newcommand{\dee}{\mathbb{D}}

\newcommand{\smoo}{\mathcal{C}}

\newcommand{\holo}{{\sf Hol}}

\newcommand{\di}{{\sf diam}}
\newcommand{\are}{{\sf area}}

\newcommand{\bcdot}{\boldsymbol{\cdot}}

\newcommand{\lrarw}{\longrightarrow}

\newcommand{\corr}{\varGamma}
\newcommand{\acorr}{\varGamma^\dagger}
\newcommand\aGa[1]{\Gamma^\dagger_{{#1}}}

\newcommand{\FS}{\omega_{FS}}
\newcommand{\weakST}{\xrightarrow{\text{weak${}^{\boldsymbol{*}}$}}}
\newcommand\Add[1]{\sum_{{#1}}\nolimits^{\prime}}
\newcommand{\IVar}{\Gamma^\bullet}

\newcommand\newfr[2]{{}^{\raisebox{-2pt}{$\scriptstyle {#1}$}}\!/_{\raisebox{2pt}{$\scriptstyle {#2}$}}}

\newcommand{\Lam}{\boldsymbol{\Lambda}}
\newcommand{\gen}{\mathscr{G}}
\newcommand\rep[1]{{\sf R}({#1})^\bullet}
\newcommand\Repp[1]{\mathscr{R}({#1})^\bullet}
\newcommand{\blot}{\bullet}
\newcommand\bran[2]{{}^{{#1}}\gamma_{{#2}}}
\newcommand\chaf[1]{\chi_{\raisebox{-2pt}{$\scriptstyle {{#1}}$}}}
\newcommand{\wt}{\widetilde}
\newcommand{\muLeb}{\mu_{\raisebox{-2pt}{$\scriptstyle{{\rm Leb}}$}}}

\newcommand{\cplx}{\mathbb{C}} 

\newcommand{\rea}{\mathbb{R}}
\newcommand{\pro}{\mathbb{P}^1}

\begin{document}

\title[Correspondences related to rational semigroups]{Holomorphic correspondences related to \\
finitely generated rational semigroups}

\author{Gautam Bharali}
\address{Department of Mathematics, Indian Institute of Science, Bangalore 560012, India}
\email{bharali@math.iisc.ernet.in}

\author{Shrihari Sridharan}
\address{Indian Institute of Science Education \& Research, Thiruvananthapuram 695016, India}
\email{shrihari@iisertvm.ac.in}

\thanks{The first author is supported in part by a Swarnajayanti Fellowship (Grant no.\;DST/SJF/MSA-02/2013-14)
and by a UGC Centre for Advanced Study grant}


\begin{abstract} 
In this paper, we present a new technique for studying the dynamics of a finitely generated
rational semigroup. Such a semigroup can be associated naturally to a
certain holomorphic correspondence on $\pro$. Results on the iterative dynamics
of such a correspondence can now be applied to the study of the rational semigroup. We
focus on an invariant measure for the aforementioned correspondence\,---\,known
as the equilibrium measure. This confers some advantages over many of the known techniques 
for studying the dynamics of rational
semigroups. We use the equilibrium measure to analyse the distribution
of repelling fixed points of a finitely generated rational semigroup, and to derive a sharp bound
for the Hausdorff dimension of the Julia set of such a semigroup.
\end{abstract}
\maketitle

\section{Introduction and Statement of Results}\label{S:intro}

Given two compact complex manifolds $X_1$ and $X_2$ of dimension $k$, a holomorphic correspondence
from $X_1$ to $X_2$ is, in essence, a relation from $X_1$
to $X_2$ that is an analytic subset of $X_1\times X_2$. Since two relations can be composed, one would
expect to have a theory of the iterative dynamics of a holomorphic correspondence. It is to this end\,---\,which
requires notions such as the degree of a correspondence, the ability to
count inverse images according to multiplicity, etc.\,---\,that one has the following definition.
 
\begin{definition}\label{D:holCorr}
With $X_1$ and $X_2$ as above, we say that $\corr$ is
a {\em holomorphic $k$-chain} in $X_1\times X_2$ if $\corr$ is a formal linear combination
of the form
\begin{equation}\label{E:stdForm}
 \corr\;= \sum_{j=1}^N m_j\Gamma_j,
\end{equation}
where the $m_j$'s are positive integers and $\Gamma_1,\dots,\Gamma_N$ are distinct irreducible
complex-analytic subvarieties of $X_1\times X_2$ of pure dimension $k$. Let $\pi_i$ denote the projection onto 
$X_i, \ i=1,2$. We say that $\corr$ is a {\em holomorphic correspondence of $X_1$ onto $X_2$} if
\begin{itemize}%[leftmargin=16pt]
  \item[$a)$] for each $\Gamma_j$ in \eqref{E:stdForm},
  $\left.\pi_1\right|_{\Gamma_j}$ and $\left.\pi_2\right|_{\Gamma_j}$ are surjective;
  \item[$b)$] for each $x\in X_1$ and $y\in X_2$, $\left(\pi_1^{-1}\{x\}\cap \Gamma_j\right)$ and
  $\left(\pi_2^{-1}\{y\}\cap \Gamma_j\right)$ are finite sets for each $j$.
 \end{itemize}
 %\item[$b)$] There exists an analytic subset $\mathfrak{S}\varsubsetneq X_2$ such that
 %$(\Gamma_j\!\setminus\!\pi_2^{-1}(\mathfrak{S}),\,X_2\!\setminus\!\mathfrak{S},\,\pi_2)$ is a
 %holomorphic covering for each $j = 1,\dots, N$. The {\em topological degree of $\corr$}
 %is defined as
 %\[
  %\sum_{j=1}^N m_j{\rm degree}\big(\left.\pi_2\right|_{\Gamma_j\setminus\pi_2^{-1}(\mathfrak{S})}\big).
 %\]
%\end{itemize}
\end{definition}

When $X_1$ and $X_2$ are compact Riemann surfaces, the condition $(b)$ above holds true automatically,
owing to holomorphicity and dimension.
A holomorphic correspondence $\corr$ determines a set-valued function, which
we denote by $F_\corr$, as follows: 
\[
X_1\supseteq A \longmapsto \bigcup\nolimits_{j=1}^N\!\pi_2\left(\pi_1^{-1}(A)\cap \Gamma_j\right).
\]
We will denote $F_\corr(\{x\})$ by $F_\corr(x)$.
\smallskip

This is the first of several articles that will use correspondences to deduce new information
on the dynamics of finitely generated rational semigroups  and revisit previously studied
phenomena under {\em weaker assumptions} than those stated in the literature. (One precise,
quantitative reason why correspondences would confer an advantage over current
techniques in studying rational semigroups is
discussed in Remark~\ref{Rem:elbrtn}.)  
\smallskip
 
A {\em rational semigroup} is a subsemigroup of $\holo(\pro; \pro)$\,---\,the semigroup with
respect to composition of holomorphic endomorphisms of $\pro$\,---\,containing no constant
functions. The study of dynamics of rational semigroups was initiated by Hinkkanen and Martin
\cite{hinkMartin:dsrf-I96, hinkMartin:spsrf96, hinkMartin:Jsrs96}, and a lot is known by now about such dynamical
systems. The aim of this work is to learn more about a finitely generated rational semigroup by associating to it a
natural holomorphic correspondence and studying the dynamics of the latter.

\begin{definition}\label{D:assoc}
Let $S$ be a finitely generated rational semigroup and let $\gen = \{g_1,\dots, g_N\}$ be a
set of generators of $S$, i.e., $S\,=\,\langle g_1,\dots, g_N\rangle$.
We call the holomorphic $1$-chain in $\pro\times \pro$
\begin{equation}\label{E:assoc}
 \corr_\gen\,:=\,\sum_{1\leq j\leq N}{\sf graph}(g_j)
\end{equation}
the {\em holomorphic correspondence associated to $(S, \gen)$}. We shall denote the
set-valued function determined by $\corr_\gen$ by $F_\gen$.
\end{definition}

$F_\gen$ is a useful book-keeping device if one is interested in the statistics of the $S$-orbits
$S\bcdot{x}$, $x\in \pro$. We shall give a formal definition of the composition of two correspondences $\corr_1$ and
$\corr_2$ in Section~\ref{S:prelims}, but we note here that the irreducible components of $\corr_2\circ\corr_1$ are
{\em precisely} those of the relation obtained by composing
the {\em relations} $\corr_2$ and $\corr_1$. Thus:
\[
 S\bcdot{x}\,=\,\bigcup\nolimits_{n\in \nat}F^{n}_\gen(x)
\]
and the finite unions $\bigcup_{0\leq n\leq M}F^{n}_\gen(x)$ give an exhaustion of $S\bcdot{x}$. Here,
$F^{n}_\gen$ is the set-valued map determined by the $n$-fold iterate of $\corr_\gen$.
\smallskip

One can construct useful invariant measures\,---\,which are analogous to the constructions of
Brolin \cite{brolin:isirf65} or Lyubich \cite{lyubich:epreRs83} for rational maps\,---\,for
holomorphic correspondences. Definition~\ref{D:assoc} provides a bridge between such a measure
and the study of a finitely generated semigroup.
In our case, we wish to exploit a measure constructed by Dinh and
Sibony \cite{dinhSibony:dvtma06}. Suppose $X$ is a compact Riemann surface, and $\corr$
a holomorphic correspondence of $X$ onto itself. Let $d_1(\corr)$ be the generic number
of inverse images of $\corr$ and let $d_0(\corr)$ be the generic number of forward images, both
counted according to multiplicity (see Section~\ref{S:prelims}). Dinh--Sibony show that regular
Borel measures can be pulled back under a correspondence: for such a measure $\mu$,
denote the pull-back by $F_\corr^*\mu$ (refer to \cite[Section~2.4]{dinhSibony:dvtma06} for details).
One of the main results of \cite{dinhSibony:dvtma06}, applied to the setting of compact Riemann
surfaces, states that when $d_1(\corr) > d_0(\corr)$, there exist a polar set $\excep\varsubsetneq X$
and a regular Borel probability measure $\mu_\corr$ such that
\begin{equation}\label{E:asymp1}
 \frac{1}{d_1(\corr)^n}F_{\corr^{\circ n}}^*(\delta_x) \weakST \mu_\corr \;\; \text{as $n\to \infty$,} \;\; 
 \forall x\in X\setminus\excep.
\end{equation}
Following Dinh--Sibony, we shall call the above measure the {\em equilibrium measure for the correspondence
$\corr$}. We must point out here that, under certain constraints on the generating sets $\gen$ of
a finitely generated rational semigroup $S$, the measure $\mu_\corr$ for the correspondence \eqref{E:assoc}
was discovered by Boyd \cite{boyd:imfgrs99}\,---\,albeit not in the context represented by \eqref{E:asymp1}.
Boyd's construction requires $S$ to be such that there is a
set of generators $\gen = \{g_1,\dots, g_N\}$ such that $\deg(g_j)\geq 2 \ \forall j = 1,\dots, N$. This is
necessitated by certain expansivity considerations; see \cite[Sections~3~\&~6]{boyd:imfgrs99}. One of the
advantages of using correspondences is that one sees\,---\,with considerable clarity\,---\,that the condition
\begin{equation}\label{E:key_cond}
 d_1(\corr_\gen)\,>\,d_0(\corr_\gen), \; \; \; \text{equivalently} \; \; \; 
 \sum\nolimits_{1\leq j\leq N}\deg(g_j)\,>\,N,
\end{equation}
suffices to have a very useful quantitative account of the expanding features of $S$\,---\,see
Lemma~\ref{L:many_branches} for details. Now,
\eqref{E:key_cond} is also the condition under which
we have the convergence \eqref{E:asymp1}. Thus, one has the analogue of Boyd's measure for a finitely generated
rational semigroup $(S, \gen)$ such that $\deg(g_j)\geq 2$ for at least one $j$, $1\leq j\leq N$.
\smallskip

Now, for {\em any} rational semigroup $S$, recall that the 
{\em Fatou set} of $S$, denoted by $\fat(S)$, is the largest open subset of $\pro$ on which $S$ is a normal family,
and that the {\em Julia set} of $S$, denoted by $\jul(S)$, is defined as:
\[
 \jul(S)\,:=\,\pro\setminus\fat(S).
\]
Also recall:

\begin{result}[Hinkkanen--Martin, \cite{hinkMartin:dsrf-I96}]\label{R:hinkMart_rep-pts}
Let $S$ be a rational semigroup containing at least one element of degree at least $2$. Then
the set of all repelling fixed points of all the elements of $S$ is dense in $\jul(S)$. 
\end{result}

\noindent{It follows that $\jul(S)$ is the closure of the union of the Julia sets of all the
elements in $S$. One sees, thus, that repelling
fixed points are a significant feature of the dynamics of a rational semigroup. This raises the
question: what is the distribution of these fixed points in $\jul(S)$\,? To make this question precise,
we consider a finitely generated semigroup $S$. Fixing a
set of generators $\gen = \{g_1,\dots, g_N\}$, a {\em word} will refer to any composition
$g_{j_m}\circ\dots\circ g_{j_2}\circ g_{j_1}$, $m\in \zahl_+$ and $j_1,\dots, j_m\in \{1,\dots, N\}$.
For a word $w$, the expression $|w|_\gen = m$ is the shorthand for the following implication: 
\[
 |w|_\gen = m\,\Rightarrow\,\text{$\exists j_1,\dots, j_m\in \{1,\dots, N\}$ such that
 $w = g_{j_m}\circ\dots\circ g_{j_1}$.}
\]
A more precise form of the above question is:}
\vspace{1mm}

\noindent{{\bf Question.} Let $\rep{n; \gen}$ be the {\em list} of repelling fixed points, repeated
according to multiplicity, of each of the words $w\in S$ given by $|w|_\gen = n$, and let $E$ be a Borel
subset of $\jul(S)$. What fraction of $\rep{n; \gen}$ is contained in $E$ for $n$ large\,?}
\vspace{1mm}

In view of \eqref{E:key_cond}, for $S$ as in Result~\ref{R:hinkMart_rep-pts} and finitely generated: 
fixing a generating set $\gen$, the equilibrium measure for the correspondence associated to
$(S, \gen)$\,---\,denote it by $\mu_\gen$\,---\,exists (see subsection~\ref{SS:DSM}).
The measure $\mu_\gen$ provides the answer to our question.

\begin{theorem}\label{T:distrbn_rep-fixpts}
Let $S$ be a finitely generated rational semigroup and let $\gen = \{g_1,\dots, g_N\}$ be a set of
generators. Assume that $\deg(g_j)\geq 2$ for at least one $j$, $1\leq j\leq N$. Let $\mu_\gen$ denote
the equilibrium measure for the correspondence $\corr_\gen$ associated to $(S, \gen)$. 
Write
\[
 D\,:=\,d_1(\corr_\gen)\,=\,\sum\nolimits_{1\leq j\leq N}\deg(g_j).
\]
Then, for the sequence of measures $\{\mu_n\}$, $\mu_n$ as defined below, we have:
\begin{equation}\label{E:asymp2}
 \mu_n\,:=\,\frac{1}{D^n}\sum\nolimits_{x\in \rep{n;\,\gen}}\!\delta_x\weakST \mu_\gen \;\; \text{as $n\to \infty$.}
\end{equation}
\end{theorem} 

\begin{remark}
We must note that $\mu_n$ is not, in general, of mass $1$. But it is not hard to show that
$\sharp\rep{n; \gen}\leq  D^n + N^n$, whence, for large $n$, $\mu_n(E)$ does give a very good
estimate of the fraction on $\rep{n\,\gen}$ contained in $E$, for $E$ as in 
the question above.
\end{remark}

The earliest investigations into the invariant-measure aspects of rational semigroups were
undertaken by Sumi  \cite{sumi:HdJshrs98, sumi:ds-hs-hrsgsp01}\,---\,following
the perspectives introduced by Sullivan \cite{sullivan:cds83}. For a finitely-generated
rational semigroup $S$, he defined $\delta$-conformal measures and 
$\delta$-subconformal measures in analogy with the concepts in \cite{sullivan:cds83}. He showed that
under certain restrictions on $S$, the Hausdorff dimension of $\jul(S)$ is
the unique number $\delta$ such that a $\delta$-conformal measure for $S$ exists. These
results give a connection between two very different aspects of the dynamics of $S$, but they
do not yield explicit estimates for the Hausdorff dimension of $\jul(S)$. 
\smallskip

Working with the equilibrium measure $\mu_\gen$ and the use of the formalism of
correspondences gives effective estimates. Our second theorem illustrates the claim that the
use of correspondences allows us to study aspects of $\jul(S)$ under {\em weaker
assumptions on $S$ than those assumed in the literature} (and with estimates that are sharp).
To elaborate, we shall first present our next result,  and then compare it with what is
currently known.
\smallskip

Some observations are in order before we state our next result. In what follows, we shall use
$\dim_H$ to denote the Hausdorff dimension. %The examples in \cite{hinkMartin:dsrf-I96} show that there
%exist finitely generated rational semigroups whose Julia sets are proper subsets of $\pro$ but have non-empty
%interiors. Given this, it is not hard to construct rational semigroups $S = \langle g_1, g_2,\dots, g_N\rangle$
%such that $\jul(S)\varsubsetneq \pro$ and $\jul(S)\cup g_1(\jul(S))\cup\dots\cup g_N(\jul(S)) = \pro$.
Suppose a rational semigroups $S = \langle g_1, g_2,\dots, g_N\rangle$ satisfies the
condition:
%\begin{equation}\label{E:J_small}
\[
 \jul(S)\cup g_1\big(\jul(S)\big)\cup\dots\cup g_N\big(\jul(S)\big)\,\varsubsetneq\,\pro.
\]
%\end{equation}
%In stating a result that provides a lower bound for $\dim_H(\jul(S))$, this is a very {\bf natural} condition,
%since we would know that $\dim_H(\jul(S)) = 2$ whenever the condition \eqref{E:J_small} fails
%(indeed, conditions that imply \eqref{E:J_small} are implicit in many of the previous results on the subject of
%$\dim_H(\jul(S))$\,---\,see below).
Then, we can choose projective coordinates on $\pro$ such that the point at infinity lies outside
$\jul(S)\cup\,g_1(\jul(S))\cup\dots\cup\,g_N(\jul(S))$. Hence, for such a semigroup $S$, 
there is no loss of generality in assuming:
\begin{equation}\label{E:J_in-C}
 \text{$\jul(S)\subset \cplx$ and $g_j$ has no poles in $\jul(S)$, $1\leq j\leq N$.}
\end{equation}
This observation is relevant to the statement of our next theorem.  

\begin{theorem}\label{T:hausDim}
Let $S$ be a finitely generated rational semigroup and let $\gen = \{g_1,\dots, g_N\}$ be a set of
generators. Assume that $\deg(g_j)\geq 2$ for at least one $j$, $1\leq j\leq N$, and that, for each
$j$, $g_j$ has no critical points in $\jul(S)$. 
\begin{itemize}[leftmargin=16pt]
 \item[$a)$] If $\jul(S)\cup\,g_1(\jul(S))\cup\dots\cup\,g_N(\jul(S)) = \pro$, then
 $\dim_H(\jul(S)) = 2$.
 \item[$b)$] If $\jul(S)\cup\,g_1(\jul(S))\cup\dots\cup\,g_N(\jul(S))\varsubsetneq \pro$, then, assuming
 the condition \eqref{E:J_in-C}, write
 \[
  M\,:=\,\sup\big\{|g_j^\prime(\xi)| : \xi\in \jul(S), \ j = 1,\dots, N\big\}.
 \]
 Then
 \begin{equation}\label{E:hausDim_loEst}
  \dim_H(\jul(S))\,\geq\,\frac{\log\big(\max_{1\leq j\leq N}\deg(g_j)\big)}{\log(M)}.
 \end{equation}
 This inequality is sharp in that, for each $N\in \zahl_+$, there is a
 semigroup $S_N$ with $N$ generators, and with the above properties, for which \eqref{E:hausDim_loEst}
 holds as an {\em equality}.
\end{itemize}
\end{theorem}

\begin{remark}
%The substantive assumption in the above theorem is the one about the critical points of $g_j$, $1\leq j\leq N$.
Under the circumstances of part~$(b)$, one can always assume \eqref{E:J_in-C} by choosing suitable coordinates\,---\,as
discussed above. Also observe: in view of Result~\ref{R:hinkMart_rep-pts},
\eqref{E:hausDim_loEst} tells us that
$\dim_H(\jul(S)) > 0$ for $S$ with the stated properties.
\end{remark}

We must mention here that this
is not the only work wherein one uses an {\em auxiliary} dynamical system\,---\,the iteration of $\corr_\gen$ in our
case\,---\,to study finitely generated rational semigroups. In \cite{sumi:spmrfgrs00}, Sumi studies a
certain class of rational semigroups by associating to them a natural skew-product map. This is 
the class of finitely generated rational semigroups wherein each $S$
admits a generating set $\gen = \{g_1,\dots, g_N\}$
such that:
\begin{align*}
 (*) \; \;  \text{\em EITHER} \ &\deg(g_j)\geq 2 \; \; \forall j = 1,\dots N, \\
 \text{\em OR} \ &\deg(g_j)\geq 2  \ \text{for at least one $j$, $1\leq j\leq N$ {\em and} $\fat(H)\supset 
 \jul(S)$,} \\
  &\,\text{where $H$ is the rational semigroup $H = \{f^{-1} : f\in S \ \text{and} \ \deg(f) = 1\}$.}
\end{align*}

Several of the works cited above\,---\,also see Sumi--Urba{\'n}ski
\cite{sumiUrbanski:mdJss-hrs11, sumiUrbanski:BpHders12}\,---\,provide some estimates for 
$\dim_H(\jul(S))$ under various conditions on $S$.
In most of these works, lower bounds are provided for semigroups that are expanding (see
\cite[Section~1]{sumi:dJsers05} for a definition)
and satisfy one of several forms of an {\em open set condition}. Whenever $S$ has at least one element
of degree\,$\geq 2$, all forms of the open set condition imply, given that $S = \langle g_1, g_2,\dots, g_N\rangle$, the
following condition (see \cite[Lemma~5.2]{sumi:dJsers05}, for instance):
\begin{equation}\label{E:partitn}
 (\text{Separation Condition}) \qquad\qquad 
 g_j^{-1}(\jul(S))\cap g_k^{-1}(\jul(S))\,=\,\varnothing \; \;
 \text{for $j\neq k$.}
\end{equation}
In summary, the two distinct types of conditions on the semigroup $S$, provided at least one of
its generators is of degree\,$\geq 2$, under which the aforementioned works provide a lower
bound on $\dim_H(\jul(S))$ are:
\begin{itemize}
 \item $S$ satisfies the separation condition \eqref{E:partitn} {\em AND} Condition~$(*)$ above;
 \item $S$ satisfies some form of an open set condition {\em AND} is expanding. 
\end{itemize}
The proof of Theorem~\ref{T:hausDim} involves studying the iterates of the correspondence
$\corr_\gen$ associated to $(S, \gen)$. The utility of this is that it has {\em no
need} for a condition such as \eqref{E:partitn}. As for the condition that $S$ be expanding
(in the sense of \cite{sumi:HdJshrs98, sumi:dJsers05}): it follows from the definition\,---\,see
\cite[Section~1]{sumi:dJsers05}, for instance\,---\,that there exist constants $C > 0$ and $\lambda > 1$
such that, with $(S, \gen)$ as in Definition~\ref{D:assoc}:
\[
 \text{for each word $w$,} \; \; |w|_\gen = m\,\Rightarrow\,
 \inf\nolimits_{x\in \jul(S)}\|w^\prime(x)\|_{{\rm sph}}\,\geq\,C\lambda^m,
\]
where we use $\|\bcdot\|_{{\rm sph}}$ to denote the magnitude of the derivative relative to the spherical metric.
Thus, the case wherein $S$ is 
expanding\,---\,which is addressed in
\cite{sumi:dJsers05, sumiUrbanski:mdJss-hrs11, sumiUrbanski:BpHders12}\,---\,is
subsumed by our condition in Theorem~\ref{T:hausDim}. Moreover, our bound for $\dim_H(\jul(S))$ 
%\eqref{E:hausDim_loEst}
is sharp.
\smallskip

\section{Technical Preliminaries}\label{S:prelims}

In this section, we shall define and elaborate upon a couple of concepts that were mentioned merely in passing
in Section~\ref{S:intro}. We begin with a note on language: given a complex manifold $X$, the phrase
{\em holomorphic correspondence on $X$} will refer to any holomorphic correspondence of $X$ onto itself.

\subsection{The composition of two holomorphic correspondences}\label{SS:compo}
Although the focus of this work is correspondences of the form \eqref{E:assoc}, we shall not restrict ourselves
to these forms in the following series of definitions. This is because\,---\,with $X_1 = X_2$\,($= X$) as in
Section~\ref{S:intro} and $\corr$ a correspondence of the form \eqref{E:stdForm}\,---\,it will be crucial (especially
in the proof of Theorem~\ref{T:hausDim}) to understand the need/role of the coefficients $m_j$.
\smallskip

Observe that $\corr$ gives an obvious {\em relation} on $X$: we shall denote it as $|\corr|$, where,
assuming the representation \eqref{E:stdForm}, 
\[
 |\corr|\,:=\,\cup_{j=1}^N\Gamma_j.
\]
I.e., $|\corr|$ is just the set underlying the object $\corr$. Now consider two holomorphic correspondences,
determined by the $k$-chains
\[
 \corr^1\,=\,\sum_{j=1}^{N_1} m_{1,\,j}\Gamma_{1,\,j}, \qquad 
 \corr^2\,= \,\sum_{j=1}^{N_2} m_{2,\,j}\Gamma_{2,\,j},
\]
in $X\times X$. The set underlying the composition $\corr^2\circ\corr^1$ is, in fact, the {\em classical}
composition of the relation $|\corr^{2}|$ with the relation $|\corr^{1}|$. If we denote the classical composition
by $\star$, then recall that
\begin{equation}\label{E:classic}
 |\corr^2|\star|\corr^1|\,:=\,\{(x,z)\in X\times X: \exists y\in X \ \text{s.t.} \
                                (x,y)\in |\corr^1|, \ (y,z)\in |\corr^2|\}.
\end{equation}
However, the object $\corr^2\circ\corr^1$ must include additional numerical data. To determine these,
we first note that the $k$-chains $\corr_1$, $\corr_2$ have the alternative representations
\begin{equation}\label{E:alt}
 \corr^1\,=\,\Add{1\leq j\leq L_1}\!\IVar_{1,\,j}, \qquad
 \corr^2\,=\,\Add{1\leq j\leq L_2}\!\IVar_{2,\,j},
\end{equation}
where the primed sums indicate that the irreducible subvarieties 
$\IVar_{s,\,j}, \ j=1,\dots,L_s$, are {\em not necessarily distinct} and are repeated according
to the coefficients $m_{s,\,j}$.
\smallskip

Firstly, the $k$-chain $\IVar_{2,\,l}\circ\IVar_{1,\,j}$ is determined by the following two conditions:
\begin{align}
 |\IVar_{2,\,l}\circ\IVar_{1,\,j}|\,=\,&\{(x,z)\in X\times X: \exists y\in X \ \text{s.t.} \
								(x,y)\in \IVar_{1,\,j}, \ (y,z)\in \IVar_{2,\,l}\}, \label{E:compos1}\\
 \IVar_{2,\,l}\circ\IVar_{1,\,j}\,\equiv\,&\sum\nolimits_{1\leq s\leq N(j,\,l)}\!\nu_{s,\,jl}Y_{s,\,jl}, \notag
\end{align}
where the $Y_{s,\,jl}$'s are the distinct irreducible components of the subvariety on the right-hand side of
\eqref{E:compos1} (which is the relation $|\IVar_{2,\,l}|\star|\IVar_{1,\,j}|$), 
and $\nu_{s,\,jl}\in \zahl_+$ is the generic number $y$'s as $(x,z)$ 
varies through $Y_{s,\,jl}$ for which the membership conditions on the right-hand side of 
\eqref{E:compos1} are satisfied. We now define:
\begin{equation}\label{E:compos2}
 \corr^2\circ\corr^1\,:= \ \sum_{1\leq j\leq L_1}\;\sum_{1\leq l\leq L_2}\IVar_{2,\,l}\circ\IVar_{1,\,j}.
\end{equation}

If $\corr^1$ and $\corr^2$ are two holomorphic correspondences on $X$, then it is well known that so is
$\corr^2\circ\corr^1$. This is especially obvious in the case of two correspondences of the form
\eqref{E:assoc}. Furthermore, correspondences of this form reveal why {\em it is desirable that the coefficients
$m_1,\dots, m_N$ form a part of the data} defining a correspondence in the sense of
Definition~\ref{D:holCorr}. To see this, consider a rational semigroup
$S\,=\,\langle g_1,\dots, g_N\rangle$ that is not freely generated. Suppose, for example, there exists
a relation
\[
 g_2\circ g_1\,=\,g_\mu\circ g_\nu, \; \; 1\leq \mu, \nu\leq N, \; \mu\neq 2, \ \text{and} \ \nu\neq 1.
\]
With  $\corr_\gen$ being the correspondence associated to $(S, \gen)$, $\gen = \{g_1,\dots, g_N\}$, we have
by definition:
\begin{equation}\label{E:compo}
 \corr_\gen\circ\corr_\gen = \sum_{1\leq j,\,l\leq N}{\sf graph}(g_l\circ g_j).
\end{equation}
Note that, due to the aforementioned relation, the irreducible variety ${\sf graph}(g_2\circ g_1)$ occurs with
multiplicity at least two in the above expression. We shall soon see that we {\em do want} to keep a
count of this multiplicity. This phenomenon illustrates why one requires the coefficients $m_1,\dots, m_N$,
in the sense of \eqref{E:stdForm}, to form a part of the data defining a holomorphic correspondence.  It is the
presence of these multiplicity data that necessitates a definiton that calls for more
than {\em merely} composing $|\corr^1|$ and $|\corr^2|$ as relations.

\subsection{The numbers $\boldsymbol{d_1(\corr)}$ and $\boldsymbol{d_0(\corr)}$}\label{SS:degrees}
The {\em topological degree} of a holomorphic correspondence is the generic number of preimages
of a point, counted according to multiplicity. To be more precise, we first recall that, with the
representation \eqref{E:stdForm} for $\corr$, it is classical that there is a Zariski-open set $\OM\subset X_2$ such that
$(\pi_2^{-1}(\OM)\cap \Gamma_j, \OM, \pi_2)$ is a $\Delta_j$-sheeted holomorphic covering for some
$\Delta_j\in \zahl_+$, $j = 1,\dots, N$.
The topological degree of $\corr$ is defined as
\begin{equation}\label{E:d_top}
 {\rm deg}_{top}(\corr)\,:=\,\sum_{j=1}^N m_j\Delta_j.
\end{equation}
When $\dim_{\cplx}(X_1) = \dim_{\cplx}(X_2) = 1$, ${\rm deg}_{top}(\corr)$ is abbreviated to
$d_1(\corr)$ (this notation is meant to evoke the dynamical degrees $d_i(\corr)$ of $\corr$,
$i = 0, 1,\dots, \dim_{\cplx}(X_2)$,
which were defined in \cite{dinhSibony:dvtma06} but whose definitions are {\em not essential} to
understand the results herein).
\smallskip

Let $X$ be a compact Riemann surface and let $\corr$ be a holomorphic correspondence on $X$.
Keeping in mind the issue of multiplicity discussed in the previous sub-section, we find that
\eqref{E:d_top} is the ``correct'' way to count the number of pre-images of generic points: i.e., we get the
following simple formula (here $\corr^{\circ n}$ denotes the $n$-fold iterated composition of $\corr$).

\begin{proposition}\label{P:degreeSeq}
Let $X$ be a compact Riemann surface and let $\corr$ be a holomorphic correspondence on 
$X$. Then $d_1(\corr^{\circ n})=d_1(\corr)^n \ 
\forall n\in \zahl_+$.
\end{proposition} 

\noindent{Essentially the same reasoning used for holomorphic maps applies even when $\corr$ is not
the graph of a map. Let $w\in X$ be a generic point in the sense of the discussion preceding 
\eqref{E:d_top}. Since $\corr$ has no irreducible components of the form
$\{a\}\times X$ or $X\times\{a\}$, $a\in X$, all the preimages of each such $w$\,---\,with,
perhaps, the exception of finitely many points\,---\,are generic in the above sense.
Hence, ${\rm deg}_{top}$ is multiplicative under composition of correspondences.}
\smallskip

For $X_1$, $X_2$ and $\corr$ as in Section~\ref{S:intro}, the {\em adjoint} of $\corr$ is the correspondence
(assuming, once again, the representation \eqref{E:stdForm} for $\corr$)
\begin{align}
 \acorr\,&:= \ \sum_{j=1}^{N} m_j\aGa{j}, \notag \\
 \text{where} \ \aGa{j}\,&:= \ \{(y,x)\in X_2\times X_1: (x,y)\in \Gamma_j\}. \notag
\end{align}
When $\dim_{\cplx}(X_1) = \dim_{\cplx}(X_2) = 1$, define $d_0(\corr) := d_1(\acorr)$.
\medskip

\section{Essential notations and lemmas}\label{S:ess}

This section is dedicated to various lemmas needed to prove
Theorems~\ref{T:distrbn_rep-fixpts} and \ref{T:hausDim}.
\smallskip

We will need an area-diameter comparison for $\pro$-valued maps that are
holomorphic on the open unit disc $\dee\subset \cplx$, which is due to Briend and
Duval \cite{briendDuval:dcmeeP01}.
In what follows, $D(a; r)$ is the open disc of radius $r$ centered at $a\in 
\cplx$; for a set $A\subset \pro$, its diameter $\di(A)$ is with respect to the spherical distance; and,
for any of the maps $\phi$ appearing in the following result,
\[
 \are(\phi(\OM))\,:=\,\int_{\OM}\phi^*\FS,
\]
where $\OM$ is any open subset of $\dee$,  and $\FS$ denotes the normalised Fubini--Study
form on $\pro$.

\begin{result}[paraphrasing Lemme~3.6 of \cite{dinh:sammcl05}]\label{R:are-dia}
There exists a constant $\delta_0 > 0$ such that for any injective holomorphic map
$\phi: \dee\lrarw \pro$ satisfying $\are(\phi(\dee))\leq \delta_0$, the following
holds: given any $r\in (0,1)$, there exists a constant $c(r) > 0$\,---\,which
is independent of $\phi$\,---\,such that
\[
 \di\big(\phi(D(0; r))\big) \leq c(r)\sqrt{\are\big(\phi(D(0; r))\big)}.
\]
\end{result}

\begin{remark}
Briend--Duval give a  more general form of the above result, for $\mathbb{P}^k$-valued maps,
in \cite{briendDuval:dcmeeP01}. However, we shall use the formulation given in \cite[Lemme~3.6]{dinh:sammcl05}
by Dinh.
\end{remark}

As in subsection~\ref{SS:compo},
we shall, for a part of this section, set up our notation and constructions for general holomorphic
correspondences on a compact Riemann surface $X$\,---\,even though we shall eventually discuss
correspondences of the form \eqref{E:assoc}. This allows us to better relate what we need here
to the notation in the works to which we need to refer.
\smallskip

If $\omega_X$ is a K{\"a}hler form on $X$ with $\int_X\omega_X = 1$,
it follows from the change-of-variables formula for branched coverings that for a holomorphic
correspondence $\corr$ on $X$:
\begin{equation}\label{E:integDegree}
 d_1(\corr)\,=\,\sum_{j=1}^N\,m_j\!\int_{{\sf reg}(\Gamma_j)}\big(\left. \pi_2\right|_{\Gamma_j}\big)^*\omega_X
 \; \; \; \text{and} \; \; \;
 d_0(\corr)\,=\,\sum_{j=1}^N\,m_j\!\int_{{\sf reg}(\Gamma_j)}\big(\left. \pi_1\right|_{\Gamma_j}\big)^*\omega_X;
\end{equation}
also see
\cite[equation~(3.1)]{dinhSibony:dvtma06} together with \cite[Section~3.5]{dinhSibony:dvtma06}.
\smallskip

We have almost all the ingredients needed to state a lemma that will play a key role in the proof of
Theorem~\ref{T:distrbn_rep-fixpts}. To this end, we need to fix a couple of basic notions: firstly,
given a non-empty open set $U\subset X$, we will call a map $\gamma: U\lrarw X$ {\em a regular
inverse branch of $\corr$} if there exists some $j$, $1\leq j\leq N$, and some irreducible component
$\mathfrak{C}$ of $\pi_2^{-1}(U)\cap \Gamma_j$\,---\,viewed as an analytic subset of
$X\times U$\,---\,such that
\begin{itemize}
 \item $\left.\pi_2\right|_{\mathfrak{C}}$ is a bijection onto $U$;
 \item $\gamma = \pi_1\circ\big(\left.\pi_2\right|_{\mathfrak{C}}\big)^{-1}$; and
 \item $\gamma$ is a biholomorphism.
\end{itemize}
Secondly, since we will often need to count according to multiplicity,
we must draw a distinction between lists and sets. These are the notational rules that we shall
follow:
\begin{itemize}
 \item A collection denoted by $A^\blot$ will represent a list; the objects in $A^\blot$ will be 
 repeated according to multiplicity. Also, $A$ will denote the set underlying $A^\blot$.
 \item The notation $\sharp(A^\blot)$ will denote the number of objects in $A^\blot$, counted
 according to multiplicity. Cardinality will be denoted by ${\sf Card}$.
\end{itemize}

\begin{lemma}\label{L:many_branches}
Let $\{\corr_n\}$ be a sequence of holomorphic correspondences on $\pro$ such that, for some $\eps\in (0, 1)$,
and for some open set $U\varsubsetneq \pro$ biholomorphic to $\dee$, we have that $U$ admits at least
$(1-\eps/2)d_1(\corr_n)$ regular inverse branches of $\corr_n$, counting according to multiplicity, for each $n\in \zahl_+$.
Furthermore, assume that
$\lim_{n\to \infty} d_0(\corr_n)/d_1(\corr_n) = 0$. Then:
\begin{itemize}[leftmargin=16pt]
 \item[$a)$] There exists a constant $K_\eps > 0$ such that for all $n\in \zahl_+$ and for at
 least $(1-\eps)d_1(\corr_n)$ of the aforementioned inverse branches of $\corr_n$, the areas of the images of $U$ are
 at most $K_\eps d_0(\corr_n)/d_1(\corr_n)$.
 \item[$b)$] Fix $W\Subset U$. If $\bran{n}{{}}$ is one of the inverse branches of $\corr_n$ mentioned in $(a)$, then
 \[
  \di(\bran{n}{{}}(W))\,=\,O\big(\,\sqrt{d_0(\corr_n)/d_1(\corr_n)}\,\big) \; \; \text{as $n\to \infty$}.
 \]
\end{itemize}
\end{lemma}

\begin{remark}
The above lemma may be viewed as a version\,---\,suitably reformulated for sequences of correspondences\,---\,of
a lemma in \cite[Section~1]{briendDuval:dcmeeP01}. Areas and diameters occurring in the above lemma are as
described prior to the statement of Result~\ref{R:are-dia}.
\end{remark}

\begin{proof}[The proof of Lemma~\ref{L:many_branches}]
Let
\[
 \sum_{j=1}^{N(n)} m_{n,\,j}\Gamma_{n,\,j}, \; \; \; n = 1, 2, 3,\dots
\]
be the representations, in the form \eqref{E:stdForm}, of the correspondences $\corr_n$, and let
$\fami(n)^\blot$ denote the {\em list} of inverse branches of $\corr_n$ given by our hypothesis.
Then (in the following, we abbreviate the set $\{(z,w): z = \bran{n}{{}}(w), w\in U\}$ as ${\sf gr}(\bran{n}{{}})$, and
write $\chaf{U}$ for the characteristic function of $U$)
\begin{align}
 \sum_{\bran{n}{{}}\in \fami(n)^\blot}\are(\bran{n}{{}}(U))\,&=\,
 \sum_{\bran{n}{{}}\in \fami(n)^\blot}\int_{U}\bran{n}{{}}^*\FS \notag \\
 &=\,\sum_{\bran{n}{{}}\in \fami(n)^\blot}\int_{{\sf gr}(\bran{n}{{}})}
 	\left(\left. \pi_2\right|_{{\sf gr}(\bran{n}{{}})}\right)^*(\chaf{U})
 	\left(\left. \pi_1\right|_{{\sf gr}(\bran{n}{{}})}\right)^*(\FS) \notag \\
 &\leq\,\sum_{1\leq j\leq N(n)}\!\!\!\!m_{n,\,j}\!\int_{{\sf reg}(\Gamma_{n,\,j})}
 	\big(\left.\pi_2\right|_{\Gamma_j}\big)^*(\chaf{U})\,\big(\left. \pi_1\right|_{\Gamma_j}\big)^*(\FS) \notag \\
 &\leq\,\sum_{1\leq j\leq N(n)}\!\!\!\!m_{n,\,j}\!\int_{{\sf reg}(\Gamma_{n,\,j})}
 	\big(\left. \pi_1\right|_{\Gamma_j}\big)^*(\FS)\,=\,d_0(\corr_n). \label{E:area_bran}
\end{align}
Here, the first inequality above is due to the fact that the union of the graphs of every $\bran{n}{{}}\in \fami(n)$ is
a subset of $(X\times U)\cap |\corr_n|$. The final equality above is one of the identities in \eqref{E:integDegree}.
Let $M(n)$ denote the number of branches $\bran{n}{{}}\in \fami(n)^\blot$ such that
\[
 \are(\bran{n}{{}}(U))\,\geq\,2\eps^{-1}d_0(\corr_n)/d_1(\corr_n).
\]
By \eqref{E:area_bran}, we have
\[
 \frac{2}{\eps} \ \frac{d_0(\corr_n)}{d_1(\corr_n)} \ M(n)\,\leq\,d_0(\corr_n),
\]
from which $(a)$ follows, with $K_\eps = 2\eps^{-1}$.
\smallskip

Let us fix $W\Subset U$ and fix a Riemann mapping $\psi : \dee\lrarw U$. Let $\overline{\fami}(n)^\blot$ denote the
list of inverse branches of $\corr_n$ given by $(a)$, and let $\delta_0$ be as given by Result~\ref{R:are-dia}.
By the area bound given by $(a)$ and by our assumption on $d_0(\corr_n)/d_1(\corr_n)$, there exists a number
$N_\eps\in \zahl_+$ such that
\[
 \are(\bran{n}{{}}\circ\psi(\dee))\,\leq\,\delta_0 \quad 
 \forall\,\bran{n}{{}}\in \overline{\fami}(n) \; \text{and} \; \forall n\geq N_\eps.
\]
Let $r\in (0, 1)$ be such that $\psi^{-1}(W)\subset D(0;r)$. By the above inequality, we are in a position to
apply Result~\ref{R:are-dia} to $\bran{n}{{}}\circ\psi$ for each $\bran{n}{{}}\in \overline{\fami}(n)$ for $n$ sufficiently
large. With $c(r)$ as given by Result~\ref{R:are-dia}, we get
\[
 \di(\bran{n}{{}}(W))\,\leq\,c(r)\sqrt{K_\eps}\sqrt{d_0(\corr_n)/d_1(\corr_n)} \quad
 \forall\,\bran{n}{{}}\in \overline{\fami}(n) \; \text{and} \; \forall n\geq N_\eps,
\]
which establishes $(b)$. 
\end{proof}

We now turn to finitely generated rational semigroups.  Let $\gen = \{g_1,\dots, g_N\}$ be a
set of generators of such a semigroup, and let $\corr_\gen$ be the associated holomorphic
correspondence. The latter is a special case of correspondences of the form
\begin{equation}\label{E:special}
 \corr\,=\,\sum_{1\leq j\leq N}{\sf graph}(g_j),
\end{equation}
where each $g_j$ is a rational map, $g_1,\dots, g_N$ {\bf not necessarily distinct.}
We call a point in $\pro$ a {\em critical value} of
$\corr$ if it is a critical value of some $g_j$, $1\leq j\leq N$. We write
\[
 {\sf crit}(\corr)\,:=\,\text{the set of all critical values of $\corr$.}
\]
We define ${\sf crit}(\corr^{\circ n})$ in an analogous way. Finally, write
\[
 C_n\,:=\,\bigcup_{j = 1}^n{\sf crit}(\corr^{\circ j}).
\]
We can now state:

\begin{lemma}\label{L:branch-mult}
Let $\corr$ be a holomorphic correspondence of the form \eqref{E:special}.
Let $U\varsubsetneq \pro$ be a simply connected open set and assume that $U\cap C_l = \varnothing$
for some $l\in \zahl_+$. Then, counting according to multiplicity, $\corr^{\circ n}$ admits at
least
\[
 d_1(\corr)^n\bigg[1 - \left(\frac{N}{d_1(\corr)}\right)^l\!\!\tau(N+1)\bigg]
\]
regular inverse branches, where $\tau = {\sf Card}(C_1)$.
\end{lemma}
\begin{proof}
A substantial part of the proof of 
the above statement occurs at the beginning of Section~6.8 of the work of Boyd
discussed in Section~\ref{S:intro}: i.e., in  \cite{boyd:imfgrs99}. Write
\begin{align*}
 \sigma_n\,:=&\,\text{the number, counting according to multiplicity, of regular} \\
 &\,\text{inverse branches of $\corr^{\circ n}$}.
\end{align*}
Since $d_1(\corr)$ is the generic number of inverse images under $\corr$, counted according to
multiplicity, the result is obvious for $n\leq l$.
The following inequality obtained by Boyd\,---\,by a variation of the argument by Lyubich in
\cite[Proposition~4]{lyubich:epreRs83}\,---\,does not require his condition that $\deg(g_j)\geq 2$
for every $j = 1,\dots, N$:
\[
 \sigma_n\,\geq\,d_1(\corr)^n - \tau N^l\sum_{j = 1}^{n- l}d_1(\corr)^jN^{n-l-j}, \; \; \; n\geq l+1.
\]
We are done if $\deg(g_j) = 1$ for each $j = 1,\dots, N$. Thus, assume that $\deg(g_j)\geq 2$
for some $j = 1,\dots, N$. Then for the right-hand side of the above inequality, we have
\begin{align*}
 d_1(\corr)^n - \tau N^l\sum_{j = 1}^{n - l}d_1(\corr)^jN^{n-l-j}\,&\geq\,
 d_1(\corr)^n\bigg[1 - \left(\frac{N}{d_1(\corr)}\right)^l\!\!\frac{\tau}{1- N/d_1(\corr)}
  \bigg] \\
  &\geq\,d_1(\corr)^n\bigg[1 - \left(\frac{N}{d_1(\corr)}\right)^l\!\!\tau(N+1)\bigg].
\end{align*}
\end{proof}
\smallskip

\section{The proof of Theorem~\ref{T:distrbn_rep-fixpts}}\label{S:distrbn_rep-fixpts}

This section is divided into two subsections. Before we give the proof of Theorem~\ref{T:distrbn_rep-fixpts},
we must take a closer look at the equilibrium measure of Dinh--Sibony. This will be the focus of the first subsection. In that
subsection, we shall also present a lemma involving the equilibrium measure, which will be needed in our
proof of Theorem~\ref{T:distrbn_rep-fixpts}.
\smallskip

\subsection{More about the  equilibrium measure of Dinh--Sibony}\label{SS:DSM}
Since some readers might not be familiar with the results of \cite{dinhSibony:dvtma06}\,---\,which
address a much more general (and multi-dimensional) set-up than ours\,---\,we shall not merely
refer the reader to the relevant sections of \cite{dinhSibony:dvtma06}. We shall paraphrase the relevant results,
adapting them to the following objects:  the normalised Fubini--Study form $\FS$ on $\pro$, and to holomorphic
correspondences on $\pro$. We will need two results to establish \eqref{E:asymp1}
above, which is the key fact about the equilibrium measure.

\begin{result}[Th{\'e}or{\`e}me~5.1 of \cite{dinhSibony:dvtma06} paraphrased for
correspondences on $\pro$]\label{R:1stRes}
Let $\corr_n$, $n\in \zahl_+$, be holomorphic correspondences on $\pro$. Suppose that the series
\[
\sum\nolimits_{n\in \zahl_+}(d_0(\corr_1)/d_1(\corr_1))\dots(d_0(\corr_n)/d_1(\corr_n))
\]
converges. Then, there exists a regular Borel probability measure $\mu$ such that
\[
 d_1(\corr_1)^{-1}\!\!\dots d_1(\corr_n)^{-1}F^*_{\raisebox{-1pt}{\!$\scriptstyle \corr_n\circ\dots\circ\corr_1$}}(\FS)
 \weakST \mu \;\; \text{as measures, as $n\to \infty$.}
\]
The measure $\mu$ places no mass on polar sets.
\end{result}

\begin{result}[a part of Th{\'e}or{\`e}me~1.2 of \cite{dinhSibony:dvtma06} paraphrased for
correspondences on $\pro$]\label{R:2ndRes}
Let $\corr_n$, $n\in \zahl_+$, be holomorphic correspondences on $\pro$. Suppose
$\sum_{n\in \zahl_+}d_0(\corr_n)/d_1(\corr_n)$ converges. Then, there exists a polar set $\excep\varsubsetneq \pro$
such that for any $w\in \pro\setminus\excep$,
\[
 d_1(\corr_n)^{-1}\!\big(F^*_{\raisebox{-1pt}{\!$\scriptstyle \corr_n$}}(\FS) -
 F^*_{\raisebox{-1pt}{\!$\scriptstyle \corr_n$}}(\delta_w)\big)
 \weakST 0 \;\; \text{as $n\to \infty$.}
\]
\end{result}

Both results involve the  notion of the pullback of a measure by a correspondence. We shall present
a {\em very brief} discussion\,---\,focused on our needs at hand\,---\,on how holomorphic correspondences act on currents.
\smallskip

To see that the pullbacks in the above results are measures, we will leave it to the reader to work
out\,---\,in analogy with the discussion that we present below\,---\,that for any {\em volume form} $\omega$,
$F_\corr^*(\omega)$ is a positive measure. Let us now define $F_\corr^*(\delta_w)$. The {\em formal}
prescription for the pullback of any current $T$ on $X$ of bidegree $(p,p)$, $p = 0, 1$, is:
\begin{equation}\label{E:pullback}
F_{\corr}^*(T)\,:=\,(\pi_1)_*\left(\pi_2^*(T)\wedge [\corr]\right).
\end{equation}
$\corr$ detemines a current of bidimension $(1,1)$ via the currents of integration given by its constituent
subvarieties $\Gamma_j$. We denote this current by $[\corr]$. The above prescription is meaningful only for 
certain types of currents. In the case of measures: we use the fact that one can define the pullback of 
a measure by a submersive mapping between manifolds. To begin with, pick a point $w$ from the
Zariski-open set $\OM$ defined in subsection~\ref{SS:degrees} (with $X_1 = X_2 = X$, a compact Riemann
surface). The prescription \eqref{E:pullback} gives
\[
 \langle F_{\corr}^*(\delta_w), \varphi \rangle\,\equiv_{{\rm by \ duality}}\,\left\langle \pi_2^*(\delta_w)\wedge [\corr],
 		\pi_1^*\varphi \right\rangle\,:=\,\sum_{j=1}^N m_j\big\langle(\left.\pi_2\right|_{\Gamma_j})^*(\delta_w),
 		\pi_1^*\varphi \big\rangle .
\]
Since, by our description of $\Omega$, $w$ is a regular value of $\left.\pi_2\right|_{\Gamma_j}$ for each
$j = 1,\dots, N$, the pullback measures above are classically known, and\,---\,in our particular case\,---\,give:
\begin{equation}\label{E:pullbackDirac1}
 \sum_{j=1}^N m_j\big\langle(\left.\pi_2\right|_{\Gamma_j})^*(\delta_w), \pi_1^*\varphi
 \big\rangle\,=\,\sum_{j=1}^N m_j\negthickspace\negthickspace\sum_{z:(z,w)\in\Gamma_j}
 \negthickspace\varphi(z)\,=:\,\Lam[\varphi](w) \; \; \text{(provided $w\in \OM$)}.
\end{equation}
For any fixed continuous function $\varphi$, $\Lam[\varphi]$ extends
continuously to each $w\in X\setminus\OM$. We shall still denote this continuous extension
of the middle term of \eqref{E:pullbackDirac1} as  $\Lam[\varphi]$, and use it to {\bf define}
$F^*_\corr(\delta_w)$: i.e.,
\begin{equation}\label{E:pullbackDirac2}
 \langle F^*_\corr(\delta_w), \varphi \rangle\,:=\,\Lam[\varphi](w) \; \; \forall w\in X, \; \;
						\forall \varphi\in \smoo(X).
\end{equation}
Thus, $F^*_\corr(\delta_w)$ is a measure supported on $F_{\acorr}(w)$. Now, combining 
Results~\ref{R:1stRes} and \ref{R:2ndRes} with Proposition~\ref{P:degreeSeq}\,---\,taking
$\corr_1=\corr_2=\corr_3=\dots=\corr$ in
Result~\ref{R:1stRes} and $\corr_n = \corr^{\circ n}$, $n = 1, 2, 3,\dots$, in Result~\ref{R:2ndRes}\,---\,we
get the following:

\begin{fact}\label{F:DSM}
Let $\corr$ be a holomorphic correspondence on $\pro$ such that $d_0(\corr) < d_1(\corr)$.
There exist a polar set $\excep\varsubsetneq \pro$  and a regular Borel probability measure $\mu_\corr$
such that for each $w\in \pro\setminus\excep$ 
\[
 d_1(\corr)^{-n}F_{\corr^{\circ n}}^*(\delta_w) \weakST \mu_\corr \;\; \text{as measures, as $n\to \infty$}.
\]
The measure $\mu_\corr$ places no mass on polar sets.
\end{fact}

\noindent{We must state that {\em we do not claim} the above fact as an original result. It is already
well known to those who have studied \cite{dinhSibony:dvtma06}; but a precise statement in the above form
is hard to find in \cite{dinhSibony:dvtma06}.} 
\smallskip

The following lemma will be useful in the proof of Theorem~\ref{T:distrbn_rep-fixpts}

\begin{lemma}\label{L:simpConn}
Let $\corr$ be any holomorphic correspondence with the properties stated in Fact~\ref{F:DSM} above,
and let $\mu_\corr$ be the measure associated to $\corr$. Let  $C = \{w_0, w_1, w_2,\dots\}$
be a countable subset of $\pro$. Given any $M\in \nat$ and $\eps > 0$, we can find a
simply-connected domain $U \equiv U(\eps, M)$ that is biholomorphic to $\dee$ such that:
\begin{itemize}
 \item $U\cap \{w_\nu: 0\leq \nu \leq M\} = \varnothing$, and
 \item $\mu_\corr(U) > 1-\eps$.
\end{itemize}
\end{lemma}
\begin{proof}
%Since the above is a variant of a fairly standard type of result, we shall be brief. I.e., we shall merely give
%an indication\,---\,and be sparing with the technical details\,---\,of some of the aspects of our argument. 
Since $\mu_\corr$ does not place any mass on polar sets (and is a finite measure), we can find a simply-connected
neighbourhood $D_\eps$ of $w_0$ with $\overline{D}_\eps\cap \{w_\nu: 1\leq \nu \leq M\} = \varnothing$ such that
\begin{equation}\label{E:opener}
 \mu_\corr(\pro\setminus\overline{D}_\eps)\,=\,1-\eps/2.
\end{equation}
If $M = 0$, then $U := (\pro\setminus\overline{D}_\eps)$ has the desired properties, and we are done.
\smallskip

For the remainder of this proof, we will assume that $M\geq 1$. 
If we set 
\[
V\,:=\,\pro\setminus\big(\overline{D}_\eps\cup \{w_\nu : 1\leq \nu \leq M\}\big),
\]
then by \eqref{E:opener}, $\mu_\corr(V) = 1-\eps/2$. Without loss of generality,
we may assume that $w_0$ is the point at infinity, whence
$({\sf supp}(\mu_\corr)\setminus\overline{D}_\eps)\Subset \cplx$. Let $\Lb$ denote the two-dimensional
Lebesgue measure on $\cplx$ and let
\[
 \mu_\corr\,=\,\mu_s + \muLeb
\]
denote the Lebesgue decomposition of $\mu_\corr$ such that $\mu_s\perp \Lb$ and 
$\muLeb\ll \Lb$. Let $Q_z(r)$ denote the closed square in $\cplx$ with edges parallel to
the real and imaginary axes of $\cplx$, centre $z$ and
edge-length $r$. Then, we know that (see \cite[Theorem~10.48]{wheeZyg:mi1977}, for instance)
\begin{equation}\label{E:shwink}
 \lim_{r\to 0^+}\frac{\mu_s(Q_z(r))}{\Lb(Q_z(r))}\,=\,0 \; \; \text{for $\Lb$-a.e. $z$.}
\end{equation}
Let $G$ denote the set on which the equality in \eqref{E:shwink} holds and let $B := \cplx\setminus G$.
\smallskip

Note that $\{w_\nu: 1\leq \nu \leq M\}$ need not be disjoint from $B$. Since $\mu_\corr$ places no
mass on polar sets, $\mu_s(\{w_\nu\}) = 0$ for $1\leq \nu\leq M$. Thus, we can find mutually disjoint
closed squares $S_1,\dots, S_M$, with $S_\nu$ centred at $w_\nu$, that are so small that
$\mu_s(S_\nu) < \eps/4M$, $1\leq \nu\leq M$. Now note that $B$ has empty interior. Thus, we can
find a piecewise smooth path $\sigma : [0,1]\lrarw \cplx$ such that, writing $\mathcal{A} := \sigma([0, 1])$,
\begin{itemize}
 \item $\sigma(\nu/M) = w_\nu$, $1\leq \nu\leq M$;
 %\item $\sigma\big(\,(\newfr{\nu-1}{M}, \ \nu/M)\,\big)\subset G$, $1\leq \nu\leq M$;
 \item $V\setminus\mathcal{A}$ is simply connected; and
 \item $\mathcal{A}\setminus \big(\cup_{1\leq \nu\leq M}S_\nu\big) = \sqcup_{1\leq \nu\leq M}A_\nu$,
 where each $A_\nu$ is a sub-arc of $\mathcal{A}$ lying in $G$.
\end{itemize}
Let us set $U := (V\setminus\mathcal{A})$. Now, by \eqref{E:opener} and the fact that $\muLeb\ll \Lb$ it follows that
\begin{equation}\label{E:goodeq}
 \muLeb(U)\,\geq\,\muLeb(V)-\muLeb(\mathcal{A})\,=\,\muLeb(V)
 -\int\nolimits_{\mathcal{A}}1\,d\muLeb\,=\,\muLeb(V).
\end{equation}
In view of \eqref{E:shwink} and the fact that $A_\nu\subset G$, $1\leq \nu\leq M$, we can find a covering of
each $A_\nu$ by small closed squares each of whose $\mu_s$-measure is $O(\delta^2)$ and such that
the number of such squares is $O(1/\delta)$ as $\delta\to 0^+$. By this\,---\,and by the fact that
$\mu_s(S_\nu) < \eps/4M$, $1\leq \nu\leq M$\,---\,we can show that $\mu_s(\mathcal{A}) < \eps/2$.
Then, in view of \eqref{E:goodeq}, $U$ has the required properties. 
\end{proof}

\subsection{The distribution of repelling periodic points}\label{SS:rep_pp}
The proof of Theorem~\ref{T:distrbn_rep-fixpts} follows from an auxiliary theorem stated in the
language of correspondences: i.e., Theorem~\ref{T:rep-fixpts_corr} below.
This theorem is in the same spirit as Th{\'e}or{\`e}me~4.1
of \cite{dinh:sammcl05} by Dinh. The proofs of both theorems rely on an idea
of Briend--Duval in \cite{briendDuval:dcmeeP01}. The conclusion of
Theorem~\ref{T:rep-fixpts_corr} does not require one of the hypotheses\,---\,which, designated as
(H4), concerns a certain measure supported on the postcritical set (we will not explicitly
need the definition of this set here)\,---\,of \cite[Th{\'e}or{\`e}me~4.1]{dinh:sammcl05}. However,
(H4) is very natural given that Dinh's result addresses not just iterations
but sequences of {\em general} holomorphic correspondences on a compact Riemann surface. 
This raises the question: could one prove Theorem~\ref{T:rep-fixpts_corr} by establishing that (H4) holds
and invoking \cite[Th{\'e}or{\`e}me~4.1]{dinh:sammcl05}? This is not entirely
clear; we shall not dwell upon this any further since one can very economically achieve the {\bf purpose}
that the hypothesis (H4) serves. This, essentially, is the {\em raison d'\^{e}tre} of
Lemma~\ref{L:simpConn}. The very special 
form of the correspondences
in Theorem~\ref{T:rep-fixpts_corr} allows us to circumvent many of the obstacles posed by the
topology of the very general setting of \cite[Th{\'e}or{\`e}me~4.1]{dinh:sammcl05}.
\smallskip

For a correspondence of the form \eqref{E:special}
on $\pro$, we say that a point $x\in \pro$
is a {\em repelling fixed point} if it is a repelling
fixed point for some $g_j$. For greater clarity, we will somtimes refer to this $x$ as a {\em repelling fixed point
associated to $g_j$}. If $x$ is a repelling fixed point, then note that its multiplicity is
\[
 {\sf Card}\big(\{1\leq j\leq N : x \ \text{is associated to $g_j$}\}\big).
\]
(Recall that the local intersection multiplicity of $\{(z,w)\in \pro\times\pro : z = w\}$ with ${\sf graph}(g_j)$ at $(x, x)$,
if $x$ is repelling fixed point of $g_j$, is $1$.)

\begin{theorem}\label{T:rep-fixpts_corr}
Let $\corr$ be a correspondence of the form
\[
 \corr\,= \sum_{1\leq j\leq N}{\sf graph}(g_j)
\]
on $\pro$, where each $g_j$ is a rational map, with $g_1,\dots, g_N$ not necessarily distinct, and $\deg(g_j)\geq 2$
for at least one $j$, $1\leq j\leq N$. Write
\[
 \Repp{n}\,:=\,\text{the {\em list} of repelling fixed points of $\corr^{\circ n}$ repeated according to multiplicity}.
\]
Then, for the sequence of measures $\{\mu_n\}$, $\mu_n$ as defined below, we have:
\begin{equation}\label{E:asymp3}
 \mu_n\,:=\,d_1(\corr)^{-n}\sum\nolimits_{x\in \Repp{n}}\!\delta_x\weakST \mu_\corr \;\; \text{as $n\to \infty$,}
\end{equation}
where $\mu_\corr$ is as given by Fact~\ref{F:DSM}
\end{theorem}
\begin{proof}
Before we proceed, we must clarify that we shall follow the notation discussed prior to Lemma~\ref{L:many_branches}
when working with lists. Also, let us abbreviate: $D:=d_1(\corr)$.
\smallskip

By B{\'e}zout's Theorem
for $\pro\times\pro$, see \cite[Chapter~IV, Section~2]{shafarevich:bagI94}, we have
\begin{align}
 \sharp\Repp{n}\,&\leq\,D^n + \sharp(\text{list of irreducible branches of $\corr^{\circ  n}$}) \notag \\
 &=\,D^n + N^n, \label{E:basic_count}
\end{align}
whence $0\leq {\rm mass}(\mu_n) < 2$. From this, and the fact that $\pro$ is compact, there
exists a subsequence $\{\mu_{n_k}\}$ and a Borel measure $\bar{\mu}$ such that
\begin{equation}\label{E:limm}
 \mu_{n_k}\weakST \bar{\mu} \; \; \text{as $k\to \infty$.}
\end{equation}
We shall show that for any such $\{\mu_{n_k}\}$, the associated limit measure $\bar{\mu}$ equals $\mu_\corr$.
To show this, it suffices to meet the following {\bf goal:} to show that for any
$\varphi\in \smoo(\pro; \rea)$ such that $\|\varphi\|_\infty = 1$, and for any $\eps > 0$,
\[
 \left|\,\int\nolimits_{\pro}\varphi d\bar{\mu} -
 \int\nolimits_{\pro}\varphi d\mu_\corr\,\right|\,<\,\eps,
\]
for each limit measure $\bar{\mu}$ described by \eqref{E:limm}.
\smallskip

Toward this goal, we may without confusion\,---\,and for simplicity of notation\,---\,relabel any subsequence
featured in \eqref{E:limm} as $\{\mu_n\}$. Fix an $\eps > 0$.
Let $L\in \zahl_+$ be so large that
\begin{equation}\label{E:downstream}
 \left(\frac{N}{D}\right)^{\!\raisebox{-2pt}{$\scriptstyle{L}$}}\!\!(N+1)\tau\,<\,\frac{\eps}{16},
\end{equation}
where $\tau$ is as in Lemma~\ref{L:branch-mult}. Taking
\[
 C\,=\,\cup_{n=1}^\infty{\sf crit}(\corr^{\circ n})
\]
in Lemma~\ref{L:simpConn} (it will not matter if $C$ is finite), and enumerating $C$ so that
$C_L = \{w_m\in C: 0\leq m\leq M\}$ for some
$M\in \nat$ (where $C_L$ is as introduced in Section~\ref{S:ess}),
we can find three simply-connected domains $W^\prime\Subset W\Subset U$, each biholomorphic
to $\dee$, such that
\[
 \mu_\corr(W^\prime), \ \mu_\corr(W), \ \mu_\corr(U)\,>\,1-\eps/8 \quad\text{and}
 \quad U\cap C_L\,=\,\varnothing.
\]
By \eqref{E:downstream} and Lemma~\ref{L:branch-mult} it follows that $\corr^{\circ n}$ admits at least
$(1- \eps/16)D^n$ regular inverse branches that are holomorphic on $U$, counting according to multiplicity.
Then, by Lemma~\ref{L:many_branches}, there exists a positive integer $N_1 \equiv N_1(\eps)\geq L$ such
that, for $n\geq N_1$, at least $D^n(1-\eps/8)$ (counting according to multiplicity) of the aforementioned regular
inverse branches\,---\,denote these branches by $\bran{n}{}$\,---\,satisfy
\begin{equation}\label{E:shrink}
 \di(\bran{n}{{}}(W))\,\leq\,B_{\eps}\sqrt{N^n/D^n} \qquad (\text{for $n\geq N_1$}),
\end{equation}
where $B_\eps > 0$ is a constant that depends only on $\eps$.
Let $\overline{\fami}(n)^\blot$ be the list of branches having the above property. So,
$\sharp\overline{\fami}(n)^\blot\geq (1-\eps/8)D^n$ (for $n\geq N_1$).
\smallskip

At this stage, we have obtained objects\,---\,and a few estimates about them\,---\,that will allow us to complete
our proof as in the last two paragraphs of the proof  of \cite[Th{\'e}or{\`e}me~4.1]{dinh:sammcl05}.
However, since one key point in the latter arises from a hypothesis in \cite{dinh:sammcl05} that {\em we do
not make}, we should provide a few more details.
Since $\mu_\corr$ does not
place mass on polar sets, we can find a point $w_0\in W^\prime$ such that for the measures $\mu^{w_0}_n$,
$\mu^{w_0}_n$ as defined below, we have according to Fact~\ref{F:DSM}:
\begin{equation}\label{E:DSM}
 \mu^{w_0}_n\,:=\,D^{-n}F^*_{\corr^{\circ n}}(\delta_{w_0}) \weakST \mu_\corr.
\end{equation}
Consider the list
$\wt{\fami}(n)^\blot:=\{\bran{n}{} \ \text{in} \ \overline{\fami}(n)^\blot : \bran{n}{}(w_0)\in W^\prime\}$.
By our construction of $W^\prime$ and by the size of the list $\overline{\fami}(n)^\blot$, we have:
\[
 1 - \eps/8\,\leq\,\mu_\corr(W^\prime) \longleftarrow
 \mu^{w_0}_n(W^\prime)\,\leq\,\frac{\sharp\wt{\fami}(n)^\blot}{D^n} + \frac{\eps}{8}.
\] 
Using this and \eqref{E:DSM}, we can find an integer $N_2\equiv N_2(\eps)$, $N_2\geq N_1$, such that
\begin{equation}\label{E:hit_count}
 \sharp\wt{\fami}(n)^\blot\,\geq\,(1-\eps/2)D^n \; \; \; \forall n\geq N_2.
\end{equation}

Consider, to begin with, $n\geq N_1$.
For each $\bran{n}{}$ listed in $\wt{\fami}(n)^\blot$, $\bran{n}{}(W)\cap W^\prime \neq \varnothing$.
As $W^\prime\Subset W$, it follows from \eqref{E:shrink} that there is a positive integer $N_3\equiv N_3(\eps)$
such that
\[
 \bran{n}{}(W)\Subset W \; \; \text{$\forall\,\bran{n}{}$ in $\wt{\fami}(n)^\blot$ and $\forall n\geq N_3$}.
\]
Since $W$ is biholomorphic to the unit disc, it follows from the Wolff--Denjoy theorem that,
associated to each $\bran{n}{}$ referenced above, there is a unique attracting fixed point of $\left.\bran{n}{}\right|_W$;
call it $z(\bran{n}{})$. Observe that, since $\bran{n}{}$ is a regular inverse branch, $z(\bran{n}{})$ is a repelling fixed
point of $\corr^{\circ n}$. Let us write
\[
 \bar{\nu}_n\,:=\,D^{-n}\sum\nolimits_{\bran{n}{}\in \wt{\fami}(n)^\blot}\!\!\delta_{z(\bran{n}{})}.
\]

Arguing as above, there exists a subsequence $\{\bar{\nu}_{n_k}\}$ and a Borel measure $\bar{\nu}$ such that
\begin{equation}\label{E:lim_meas}
 \bar{\nu}_{n_k}\weakST \bar{\nu} \; \; \text{as $k\to \infty$.}
\end{equation}
Fix a function $\varphi\in \smoo(\pro; \rea)$ such that $\|\varphi\|_\infty = 1$.
Then, owing to the fact that each $z(\bran{n}{})$ is a repelling fixed
point of $\corr^{\circ n}$, the count \eqref{E:basic_count}, and the estimate \eqref{E:hit_count}, we have
\[
 \left|\,\int\nolimits_{\pro}\varphi\,d\bar{\nu}_{n_k} -
 \int\nolimits_{\pro}\varphi\,d\bar{\mu}_{n_k}\,\right|\,\leq\,\frac{N^n}{D^n}+\frac{\eps}{2} \; \; \;
 \forall k: n_k\geq \max(N_2, N_3).
\]
Observe that the above estimate holds true irrespective of the subsequence $\{\bar{\nu}_{n_k}\}$
and of $\eps > 0$. Thus,
keeping in mind the goal stated at the beginning of the proof and the fact that $\eps > 0$ is
arbitrary, it suffices to meet the following {\bf revised goal:} to show that for any $\varphi\in \smoo(\pro; \rea)$
with $\|\varphi\|_\infty = 1$, and for any $\eps > 0$,
\[
 \left|\,\int\nolimits_{\pro}\varphi d\bar{\nu} -
 \int\nolimits_{\pro}\varphi d\mu_\corr\,\right|\,<\,2\eps,
\]
for each limit measure described by \eqref{E:lim_meas}.
\smallskip

Observe that there is no loss of generality in relabelling a subsequence $\{\bar{\nu}_{n_k}\}$
as $\{\bar{\nu}_n\}$ for the purposes of the revised goal. To achieve this, we introduce the auxiliary
measures
\[
 \mu^\prime_n\,:=\,D^{-n}\sum\nolimits_{\bran{n}{}\in \wt{\fami}(n)^\blot}\!\!\delta_{\bran{n}{}(w_0)}.
\]
In view of \eqref{E:DSM}, \eqref{E:hit_count} and \eqref{E:lim_meas}, we can find an
integer $N_*\equiv N_*(\eps, \varphi)$ sufficiently large that\,---\,for a fixed $\varphi$
as described above\,---\,we have:
\[
 \left|\,\int\nolimits_{\pro}\varphi\,d{\sf m}_n -
 \int\nolimits_{\pro}\varphi\,d{\sf m}^\prime_n\,\right|\,\leq\,\eps/2 \; \; \;
 \forall n\geq N_*,
\]
where the pair of measures $({\sf m}_n, {\sf m}^\prime_n)$ is either $(\mu_\corr, \mu^{w_0}_n)$ or
$(\mu^{w_0}_n, \mu^\prime_n)$ or $(\bar{\nu}_n, \bar{\nu})$. We now need to account for the pair of
measures $(\mu^\prime_n, \bar{\nu}_n)$. Since
$z(\bran{n}{}), \bran{n}{}(w_0)\in \bran{n}{}(W)$ for each $\bran{n}{}$ in $\wt{\fami}(n)^\blot$ ($n$
sufficiently large), by raising the value of $N_*$ further if necessary we can ensure, by continuity
of $\varphi$ and by \eqref{E:shrink}, that
\[
 |\varphi(\bran{{\raisebox{-2pt}{$\scriptstyle N_*$}}}{}(w_0)) -
 \varphi(z(\bran{{\raisebox{-2pt}{$\scriptstyle N_*$}}}{}))|\,<\,\eps/2 \; \; \; 
 \text{$\forall\,\bran{{\raisebox{-2pt}{$\scriptstyle N_*$}}}{}$ in $\wt{\fami}(N_*)^\blot$}.
\]
Our revised goal is achieved
through the last two inequalities by using a standard triangle-inequality argument.
\end{proof}

%It is now easy to give
\begin{proof}[The proof of Theorem~\ref{T:distrbn_rep-fixpts}]
Given a finitely generated rational semigroup $S$, having fixed a set of generators $\gen = \{g_1,\dots, g_N\}$,
we apply Theorem~\ref{T:rep-fixpts_corr} to the correspondence $\corr_\gen$, which is given by
\eqref{E:assoc}.

By (a generalisation of) the composition formula \eqref{E:compo}, it is easy to see that
\[
 \Repp{n}\,=\,\rep{n; \gen} \; \; \; \text{for $\corr = \corr_\gen$.}
\]
The measures $\mu_n$ in \eqref{E:asymp2} are special cases of the measures $\mu_n$ considered
in Theorem~\ref{T:rep-fixpts_corr}. Hence, from \eqref{E:asymp3}, Theorem~\ref{T:distrbn_rep-fixpts}
follows.
\end{proof}

\begin{remark}\label{Rem:elbrtn}
The conclusion of Theorem~\ref{T:distrbn_rep-fixpts} for the case of the semigroups $\{f^{\circ n} : n\in \nat\}$,
where $f$ is a rational map with $\deg(f)\geq 2$, follows easily from \cite[\S3, Theorem~3]{lyubich:epreRs83}.
%A natural question that arises is whether one really requires the auxiliary object $\corr_\gen$ to prove
%Theorem~\ref{T:distrbn_rep-fixpts}.
A possible approach to proving Theorem~\ref{T:distrbn_rep-fixpts} is suggested by the proof of
\cite[\S3, Theorem~3]{lyubich:epreRs83} (from which we borrow one key idea). But
the expanding features of the aforementioned $f$ play
such a significant role in the latter proof that the conclusion of
Theorem~\ref{T:distrbn_rep-fixpts} that is obtainable by a more ``direct'' approach holds only for semigroups
$\langle g_1,\dots, g_N\rangle$  such that $\deg(g_j)\geq 2$ for {\em each} $j = 1,\dots, N$. It {\em may}
be possible, with much more effort, to improve upon the latter without appealing to any auxiliary object, but the
use of the correspondence $\corr_\gen$ has certain advantages.
For instance, it reveals that the condition
\eqref{E:key_cond} endows $\corr_\gen$ with expanding features that are adequate to be useful in
applications. Here, we understand ``expanding'' to mean the existence of sufficiently many holomorphic
{\em inverse} branches of $\corr_\gen^{\circ n}$, $n = 1, 2, 3,\dots$, the diameters of whose ranges
shrink to zero as $n\to \infty$. This is the content of Lemma~\ref{L:many_branches}, which is actually
a phenomenon about general holomorphic correspondences. Of course, Lemma~\ref{L:branch-mult} plays
a facilitating role for the use of Lemma~\ref{L:many_branches}, but Lemma~\ref{L:branch-mult}
can also be stated for a more general class of correspondences.
\end{remark}

\section{The proof of Theorem~\ref{T:hausDim}}

The proof of Theorem~\ref{T:hausDim} requires the following result.

\begin{lemma}\label{L:dim_bound}
Let $\mu$ be a probability measure on a set $X\subset \cplx$. If there exist positive constants
$c, t$ and $r_0$ such that, for each $r\in (0, r_0)$,
\[
 \mu(D(z; r))\,\leq\,cr^t
\]
for every disc $D(z; r)$ with centre $z$ and radius $r$ that has non-empty intersection with $X$, then
$\dim_H(X)\geq t$.
\end{lemma}

\noindent{The above lemma is an immediate consequence of a standard result:
see \cite[Mass distribution principle, \S4.1]{falconer:fg03}, for instance. Our proof also requires
Theorem~\ref{T:rep-fixpts_corr}: specifically, that it implies that for the measure $\mu_\corr$
discussed therein, ${\sf supp}(\mu_\corr)\subseteq \jul\big(\langle g_1,\dots, g_N\rangle\big)$.
\smallskip

Theorem~\ref{T:hausDim} extends a result of Garber \cite{garber:irf78}, on the Julia set of a single rational map
of degree\,$\geq 2$,
to finitely generated semigroups.

\begin{proof}[The proof of Theorem~\ref{T:hausDim}] We first dispense with part~$(a)$ of the
theorem. Since the maps $\left.g_j\right|_{\jul(S)}$, $1\leq j\leq N$, are Lipschitz, it is impossible that 
$\jul(S)$ has empty interior but $g_j(\jul(S))$ has non-empty interior for some $j$. Hence, if
$\jul(S)\cup\,g_1(\jul(S))\cup\dots\cup\,g_N(\jul(S)) = \pro$, then it follows that $\jul(S)$ must have
non-empty interior. Hence $\dim_H(\jul(S)) = 2$.
\smallskip

The strategy for proving part~(b) is to consider a sequence of correspondences related to the semigroup $S$
whose equilibrium measures are subsets of $\jul(S)$. We shall see that a lower bound for
the Hausdorff dimension of the supports of the equilibrium measures of each of these correspondences
will give a progressively better lower bound for $\dim_H(\jul(S))$.
\smallskip
 
As in the statement of the theorem, we consider a set of generators
\[
 \gen\,=\,\{g_1,\dots, g_N\},
\]
and enumerate $\gen$ so that $\deg(g_N) = \max_{1\leq j\leq N}\deg(g_j)$. We define the
correspondences
\[
 \corr(k)\,:=\,\sum_{1\leq j\leq (N-1)}\!\!{\sf graph}(g_j) + k\bcdot{\sf graph}(g_N), \; \; \; k = 1, 2, 3,\dots
\]
Let us denote the equilibrium measure $\mu_{\corr(k)}$ as simply $\mu(k)$. 
For a fixed $k$, let $\{\mu_n(k)\}$ be the sequence of measures approximating $\mu(k)$ in the sense of
\eqref{E:asymp3}. Then:
\[
 {\sf supp}(\mu_n(k))\,=\,\{z : z \ \text{is a repelling fixed point of some word $w$ such that $|w|_\gen = n$}\}.
\]
%Furthermore, given $m, n\in \zahl_+$, it is easy to see that
%\[
 %{\sf supp}(\mu_m(k)), \ {\sf supp}(\mu_n(k))\,\subset\,{\sf supp}(\mu_{mn}(k)).
%\]
Thus, it follows from
Theorem~\ref{T:rep-fixpts_corr}, that 
\[
 {\sf supp}(\mu(k))\,\subseteq\,\overline{\{z : z \ \text{is a repelling fixed point of some $g\in S$}\}}, \; \; \; \forall k.
\]
By Result~\ref{R:hinkMart_rep-pts} we get:
\begin{equation}\label{E:supp_k}
 {\sf supp}(\mu(k))\,\subseteq\,\jul(S) \; \; \; \forall k.
\end{equation}
%\vspace{0.1mm}

\noindent{{\bf Step 1.} {\em A lower bound for $\dim_H({\sf supp}(\mu(k)))$.}}
\vspace{0.5mm}

\noindent{Fix a $k\in \zahl_+$.
We rewrite $\corr(k)$ in the form given by \eqref{E:special}, wherein the index $j$ runs from
$1$ to $(N+k-1)$ and
\[
 g_j\,=\,\begin{cases}
 		g_j, &\text{if $1\leq j\leq (N-1)$}, \\
 		g_N, &\text{if $N\leq j\leq (N+k-1)$}.
 		\end{cases}
\]
%There is no loss of generality 
%in assuming 
%(since we may conjugate $S$ by an appropriate M{\"o}bius transformation if necessary)
%that $\jul(S)\subset \cplx$ and that $g_j$ has no poles on $S$, $1\leq j\leq N$. Then, our hypothesis on the
%maps $G_j$\,---\,as defined in the statement of Theorem~\ref{T:hausDim}\,---\,translates to (for simplicity of
%notation, we shall work with $g_j$ below rather than with the purely auxiliary $G_j$):
By hypothesis, we have:
\[
 |g_j^\prime(\xi)|\,\leq\,M \; \; \; \forall \xi\in \jul(S) \; \text{and for} \; j=1,\dots, (N+k-1),
\]
where, by the compactness of $\jul(S)$, $M$ is finite.
In what follows, we shall abbreviate $\jul(S)$ to $\jul$. Set
\[
 \lambda(k)\,\:=\,\frac{\log(d_1(\corr(k))/d_0(\corr(k)))}{\log(M)}\,=\,\frac{\log((D+(k\!-\!1)\deg(g_N))/(N+k-1))}{\log(M)},
\]
where $D$ has the same meaning as in Section~\ref{S:distrbn_rep-fixpts}. Fix $t\in (0, \lambda(k))$, and abbreviate
\[
 R(k)\,:=\,\frac{D+(k\!-\!1)\deg(g_N)}{N+k-1}\,\equiv\,\frac{d_1(k)}{d_0(k)}.
\]}

For a positive number $\eps$, let us write
\[
 \jul^{\eps}\,:=\,\jul(S)^\eps\,:=\,\cup_{\xi\in \jul(S)}D(\xi; \eps), \quad \text{and} \quad
 \overline{\jul}^{\eps}\,:=\,\overline{\jul(S)^\eps}.%\overline{\cup_{\xi\in \jul(S)}D(\xi; \eps)}.
\]
By Result~\ref{R:hinkMart_rep-pts}, $1 < M < R(k)^{1/t}$. Owing to this, and to our assumption that 
$g_j$ has no critical points on $\jul$, $j = 1,\dots, N$, we can find: 
\begin{itemize}
 \item $\delta_1 > 0$ such that $g_j^\prime(\xi)\neq 0$ for every $\xi\in \jul^{2\delta_1}$, 
 $j=1,\dots,(N+k-1)$;
 \item $r(\xi) > 0$ such that $\left.g_j\right|_{D(\xi;\,r(\xi))}$ is injective, $j=1,\dots, (N+k-1)$, as
 $\xi$ varies through $\overline{\jul}^{\delta_1}$;
 \item $\eps_t > 0$ such that $|g^\prime(\xi)| < R(k)^{1/t}$ for every $\xi\in \jul^{\eps_t}$, 
 $j=1,\dots,(N+k-1)$.
\end{itemize}
Let
\[
 \delta_2\,:=\,\text{the Lebesgue number of the open cover 
 $\big\{D(\xi; r(\xi)) : \xi\in \overline{\jul}^{\delta_1}\big\}$.}
\]
Write $r_t := \min(\delta_1, \delta_2, \eps_t)/4$. We will need to work with the partition $\sqcup_{l=1}^\infty I(l,k)$
of $(0, r_t]$, where
\[
 I(l, k)\,:=\,\big(r_tR(k)^{-l/t}, r_tR(k)^{(1-l)/t}\big], \; \; \; l=1, 2, 3,\dots
\]

Since $\mu(k)$ places no mass on polar sets, by Fact~\ref{F:DSM},
we can find a point $w\in {\sf supp}(\mu(k))$\,---\,i.e., by
\eqref{E:supp_k}, $w\in \jul$\,---\,such that for the measures $\mu_n(w, k)$ as defined below, we have:
\[
  \mu_n(w, k)\,:=\,d_1(\corr(k))^{-n}F^*_{\corr(k)^{\circ n}}(\delta_{w}) \weakST \mu(k).
\]
It is easy to see from the definition of $\jul$\,---\,see \cite[Theorem~2.1]{hinkMartin:dsrf-I96}\,---\,that
$f^{-1}\{w\}\subset \jul$ for each $f\in S$. Hence, ${\sf supp}(\mu_n(w, k)) = 
(F_{\corr^{\circ n}})^\dagger(w)\subset \jul$. Let us fix $r\in (0, r_t]$ and consider a disc
$D(z; r)$ such that $D(z; r)\cap \jul\neq \varnothing$. Let $l\in \zahl_+$ be such that
$r\in I(l, k)$. Then, from the descriptions of the parameters above:
\begin{itemize}
 \item[$a)$] $D(z; r)\varsubsetneq D(x(z); 2r_tR(k)^{(1-l)/t})\subseteq 
 D(x(z); \eps_t/2)\cap D(x(z); \delta_1/2)$ for some point $x(z)$ belonging to $\jul$.
 \item[$b)$] $|g_j^\prime(\zt)|< R(k)^{1/t}$ for every $\zt\in D(z; r)$, $j=1,\dots, (N+k-1)$.
 \item[$c)$] $g_j$ is injective on $D(z; r)$ for $j=1,\dots, (N+k-1)$.
\end{itemize}
With these three assertions, we can prove the following claim, which is the central assertion of this proof.
Once again, we remind the reader that the counts made in the claim below are according to multiplicity, and
we shall distinguish between sets and lists using the notation introduced in Section~\ref{S:ess}. Also, in the
style of Section~\ref{S:intro}, let us write $(F^n)^\dagger := (F_{\corr^{\circ n}})^\dagger$.
\vspace{3mm}

\noindent{{\bf Claim.} Let $r\in (0, r_t]$ and let $D(z;r)$ be an open disc such that $D(z; r)\cap \jul\neq \varnothing$.
Let $l\in \zahl_+$ such that $r\in I(l, k)$. Then, for any $n\in \nat$,
\[
 \sharp\big((F^n)^\dagger(w)\cap D(z; r)\big)^\blot\,\leq\,\max\big(d_0(k)^n, d_1(k)^{n-l+1}d_0(k)^{(\min(n,\,l)-1)}\big).
\]}
\vspace{-1mm}

\noindent{We will prove this claim by induction on the parameter $n$. Note that the claim is trivial when
$n = 0$. Let us assume that the claim is true for $n = m$ for some $m\in \nat$. We will study what this
implies for $n = m+1$. We have two cases to analyse.}
\medskip

\noindent{{\em Case 1.} $r\in I(1, k)$}
\vspace{0.5mm}

\noindent{In this case, the claim is obvious for $n = m+1$ by the meaning of $d_1(\corr(k)) =: d_1(k)$.}
\medskip
%\pagebreak

\noindent{{\em Case 2.} $r\in I(l, k)$ and $l\geq 2$.}
\vspace{0.5mm}

\noindent{Let $p$ denote any point in the set $(F^{m+1})^\dagger(w)\cap D(z; r)$ (which might occur with
multiplicity greater than $1$). We need to define a few sets:
\begin{align*}
 \mathcal{S}(m+1)\,&:=\,\{1\leq j\leq (N+k-1) : (F^{m+1})^\dagger(w)\cap D(z; r)\cap 
 g_j^{-1}((F^m)^\dagger(w))\neq \varnothing\}; \\
 \mathcal{S}(p, m+1)\,&:=\,\{j\in \mathcal{S}(m+1) : p\in g_j^{-1}((F^m)^\dagger(w))\};
\end{align*}
and the list
\begin{multline}
 L(j, p, m)^\blot\,:=\,\text{the list of all points $\zt\in (F^m)^\dagger(w)$, repeated} \\
 \text{according to multiplicity,
 such that $p\in g_j^{-1}\{\zt\}$.}
\end{multline}}
A point $p\in (F^{m+1})^\dagger(w)$ may have to be counted with multiplicity\,$\geq 2$ precisely because
it could occur as the pre-image under $g_j$ of some point $\zt(j)\in (F^m)^\dagger(w)$ for more than
one $j$, and each $\zt(j)$ itself must be counted {\em in its turn} according to its multiplicity.
Therefore:
\[
 \sharp\big((F^{m+1})^\dagger(w)\cap D(z; r)\big)^\blot\,=\,\sum_{p\in (F^{m+1})^\dagger(w)\cap D(z;\,r)}
 \; \sum_{j\in \mathcal{S}(p,\,m+1)}\sharp L(j, p, m)^\blot.
\]
Note that the indices of the two sums above run through {\bf sets}, whereas $L(j, p, m)^\blot$ is a list.
By $(c)$ above, we have that for each relevant $p$ in the above equation and each $j\in \mathcal{S}(p, m+1)$
associated to it, $g_j$ is injective on $D(z; r)$. Using this, and interchanging the order of summation in the above equation:
\begin{equation}\label{E:1st_count}
 \sharp\big((F^{m+1})^\dagger(w)\cap D(z; r)\big)^\blot\,=\,\sum_{j\in \mathcal{S}(m+1)}
 \sharp\big((F^{m})^\dagger(w)\cap g_j(D(z; r))\big)^\blot.
\end{equation}

Now, owing to $(b)$ above, it follows from an elementary computation that
$g_j(D(z; r))\subset D(g_j(z); rR(k)^{1/t})$. Furthermore, by the definition of $\mathcal{S}(m+1)$
and as $(F^n)^\dagger(w)\subset \jul$ for each $n\in \nat$,
\[
 j\in \mathcal{S}(m+1)\,\Rightarrow\,g_j(D(z; r))\cap\jul \; \text{contains some point of $(F^m)^\dagger(w)$.}
\]
From the last two facts, we have:
\begin{equation}\label{E:touches}
  j\in \mathcal{S}(m+1)\,\Rightarrow\,D(g_j(z); rR(k)^{1/t})\cap\jul\,\neq\,\varnothing.
\end{equation}
By \eqref{E:1st_count}, we have
\[
 \sharp\big((F^{m+1})^\dagger(w)\cap D(z; r)\big)^\blot\,\leq\,\sum_{j\in \mathcal{S}(m+1)}
 \sharp\big((F^{m})^\dagger(w)\cap D(g_j(z); rR(k)^{1/t})\big)^\blot.
\]
By \eqref{E:touches}, each of the discs $D(g_j(z); rR(k)^{1/t})$ featuring in the right-hand side of the above
equation satisfies the hypothesis of our claim above. Then, from the last inequality, our inductive assumption,
and from the fact that $rR(k)^{1/t}\in I(l-1, k)$, we have
\begin{align*}
  \sharp\big((F^{m+1})^\dagger(w)\cap D(z; r)\big)^\blot\,&\leq\,d_0(k)
  \max\big(d_0(k)^m, d_1(k)^{m-(l\!-\!1)+1}d_0(k)^{(\min(m,\,(l\!-\!1))-1)}\big) \\
  &=\,\max\big(d_0(k)^{(m\!+\!1)}, d_1(k)^{(m\!+\!1)-l+1}d_0(k)^{(\min((m\!+\!1),\,l)-1)}\big).
\end{align*}

From Cases~1~and~2, our claim is true for $n=m+1$. By induction, our claim it follows for all $n\in \nat$.
\hfill $\blacktriangleleft$
\smallskip

Fix a disc $D(z; r)$, $r\in (0, r_t]$,
such that $D(z; r)\cap \jul\neq \varnothing$. From the above claim, we get that
for sufficiently large $n$
\[
 d_1(k)^{-n}F^*_{\corr(k)^{\circ n}}(\delta_w)(D(z; r))\,\leq\,\left(\frac{d_1(k)}
 {d_0(k)}\right)^{1-l}\,\leq\,\frac{R(k)}{(r_t)^t}r^t,
\]
where $l$ is the unique integer such that $r\in I(l, k)$. By the limiting behaviour of the above measures,
we see from above that we have a constant $c\equiv c(k, t)$ such that
\[
 \mu(k)(D(z; r))\,\leq\,cr^t \; \; \; \text{(provided $r\in (0, r_t]$, $r_t$ as fixed above)}.
\]
By Lemma~\ref{L:dim_bound}, and the fact that $t$ was picked arbitrarily from $(0, \lambda(k))$, we
get
\begin{equation}\label{E:HD_low-bd}
 \dim_H(\jul)\,\geq\,\dim_H({\sf supp}(\mu(k)))\,\geq\,\frac{\log((D+(k\!-\!1)\deg(g_N))/(N+k-1))}{\log(M)}.
\end{equation}
\smallskip

\noindent{{\bf Step 2.} {\em A lower bound for $\dim_H(\jul)$.}}
\vspace{2mm}

\noindent{Observe
\[
 \frac{D + (k-1)\deg(g_N)}{N+k-1}\,=\,\frac{\newfr{D}{(k-1)} + \deg(g_N)}{\newfr{N}{(k\!-\!1)} + 1}
 \lrarw \deg(g_N) \; \; \; \text{as $k\to +\infty$}.
\]
From this, \eqref{E:supp_k} and \eqref{E:HD_low-bd}, the lower bound \eqref{E:hausDim_loEst} for $\dim_H(\jul)$ follows.}
\vspace{2mm}

\noindent{{\bf Step 3.} {\em The sharpness of the inequality \eqref{E:hausDim_loEst}}}
\vspace{2mm}

\noindent{It is well known that \eqref{E:hausDim_loEst} is sharp for the semigroups $\{f^{\circ n} : n\in \nat\}$,
where $f$ is a rational map. The point that we wish to make is that \eqref{E:hausDim_loEst} is sharp for
rational semigroups having $N$ generators, for any given $N$. To this end, fix $N\in \zahl_+\setminus\{1\}$. Now let
$1 < d_1 < d_2 <\dots < d_N$ be positive integers such that
\[
 z^{d_j}\not\in \langle z^{d_1},\dots, z^{d_{j-1}}\rangle, \; \; \; j=2,\dots, N.
\]
Let $S_N := \langle z^{d_1},\dots, z^{d_N}\rangle$. For any rational semigroup $S$ having at least one element
of degree at least 2
\[
 \jul(S)\,=\,\overline{\bigcup_{g\in S}\jul(g)},
\]
which follows from Result~\ref{R:hinkMart_rep-pts}. Thus, $\jul(S_N) = \bdy\dee$. By definition,
\[
 M\,=\,\sup\big\{d_j|\xi^{d_j - 1}| : \xi\in \bdy\dee, \ j = 1,\dots N\big\}\,=\,d_N.
\]
It is also obvious that none of the maps $z\longmapsto z^{d_j}$, $j = 1,\dots, N$, has critical points on
$\bdy\dee$. It is trivial to see now that the inequality \eqref{E:hausDim_loEst} holds true for $S_N$ as an
equality.}
\end{proof}

\end{document}